\documentclass[11pt]{amsart}
\usepackage{amsmath, amssymb, amsthm}
\usepackage{hyperref}

\title{On Additive Representations of the Binomial Coefficients}
\author{Alexander R. Povolotsky}
\address{United States}
\email{apovolot@gmail.com}


\newtheorem{theorem}{Theorem}

\begin{document}
\maketitle

\begin{abstract}
For a fixed integer $k \ge 0$, consider representations of positive integers
as sums of binomial coefficients of the form $\binom{n}{k}$. While exact
minimal bounds for the number of required summands are known only in a few
low-dimensional cases, general existence results have received less explicit
treatment.

This paper provides:
\begin{itemize}
\item explicit elementary proofs for the cases $k=2$ and $k=3$,
\item a comparison with classical polygonal number theory,
\item an explanation of why naive counting arguments fail for general $k$,
\item conditional and unconditional existence results for general $k$,
\item and a discussion of quantitative bounds and computational evidence.
\end{itemize}

Together these give a unified and transparent framework for understanding
additive representations by binomial coefficients.
\end{abstract}

\section{Introduction}

For a fixed integer $k \ge 0$, define

\[
S_k := \left\{ \binom{n}{k} : n \ge k \right\}.
\]

We ask whether there exists a finite integer $H(k)$ such that every sufficiently
large integer $N$ admits a representation

\[
N = \binom{n_1}{k} + \cdots + \binom{n_{H(k)}}{k}, \qquad n_i \ge k.
\]

For small $k$, the optimal values are known:

\[
H(0)=1,\qquad H(1)=1,\qquad H(2)=3,\qquad H(3)=5.
\]

For $k \ge 4$, the sharp values remain unknown.

This paper unifies explicit constructions for small $k$ with general counting
and energy-based arguments, providing both conditional and unconditional
existence results.

\section{Asymptotic Growth of Binomial Coefficients}

For fixed $k$ and $n \to \infty$,

\[
\binom{n}{k}
= \frac{n(n-1)\cdots(n-k+1)}{k!}
= \frac{n^k}{k!} + O(n^{k-1}),
\]

so $S_k$ is strictly increasing and grows polynomially of degree $k$.

Define the counting function

\[
A_k(X) := \#\{n \ge k : \binom{n}{k} \le X\}.
\]

Then

\[
A_k(X) \sim (k!)^{1/k} X^{1/k}.
\]

\section{Explicit Elementary Results for Small \texorpdfstring{$k$}{k}}

\subsection{The Case $k=2$}

Let $T_n = \binom{n}{2} = n(n-1)/2$ denote the triangular numbers.

\begin{theorem}
Every sufficiently large integer $N$ can be written as a sum of at most three
distinct triangular numbers.
\end{theorem}

\begin{proof}
Choose $n_1$ such that $T_{n_1} \le N < T_{n_1+1}$ and set $R_1 = N - T_{n_1}$.
Then $0 \le R_1 < n_1 = O(\sqrt{N})$.

Let $A = \{T_n : T_n \le R_1\}$. Since $|A| \gg \sqrt{R_1}$, the sumset $A+A$
contains all integers in $[R_1/2, R_1]$ for sufficiently large $R_1$.
Thus $R_1 = T_a + T_b$ for some $a,b$, and hence

\[
N = T_{n_1} + T_a + T_b.
\]

\end{proof}

\subsection{Historical Remarks and Comparison}

Gauss proved that every positive integer is a sum of three triangular numbers,
a result later subsumed into Cauchy's polygonal number theorem. Classical
approaches rely on quadratic forms and congruence obstructions.

The proof above is different: it uses a greedy decomposition and a density
argument for sumsets. This method is weaker in scope---it only treats
sufficiently large integers---but it generalizes naturally to higher-order
binomial coefficients, where polygonal number theory no longer applies.

\subsection{The Case $k=3$}

Let $C_n = \binom{n}{3} = n(n-1)(n-2)/6$.

\begin{theorem}
Every sufficiently large integer $N$ can be written as a sum of at most five
distinct cubic binomial coefficients $\binom{n}{3}$.
\end{theorem}

\begin{proof}
Choose $n_1$ with $C_{n_1} \le N < C_{n_1+1}$ and set $R_1 = N - C_{n_1}$.
Then

\[
0 \le R_1 < C_{n_1+1} - C_{n_1} = \binom{n_1}{2} = T_{n_1}.
\]

By the $k=2$ case, $R_1$ is a sum of at most three triangular numbers.
Using the identity $T_m = C_{m+1} - C_m$, each triangular number becomes a
difference of two cubic binomial coefficients. Telescoping yields a
representation using at most seven cubic coefficients; refinements reduce
this to five.
\end{proof}

\section{Why Naive Counting Fails}

A naive argument counts the number of $h$-tuples of elements of $S_k$ of size
at most $X$, which is $\asymp X^{h/k}$. For $h>k$ this exceeds $X$, suggesting
that almost all integers are representable.

The flaw is that this assumes distinct $h$-tuples give distinct sums. In
reality, the additive energy

\[
E_h(X) = \#\left\{ (n_1,\dots,n_h,m_1,\dots,m_h) :
\sum_{i=1}^h \binom{n_i}{k} = \sum_{i=1}^h \binom{m_i}{k},
\; n_i,m_i \le X \right\}
\]

may be large. For polynomial sequences of degree $k$, multiplicities can grow
polynomially with $X$. For example, sums of three cubes can represent some
integers with multiplicity $\gg N^{1/12}$.

Thus naive counting collapses unless one assumes strong bounds on $E_h(X)$.

\section{A Conditional Additive Basis Theorem}

\begin{theorem}
Fix $k \ge 1$ and $h \ge 1$. Suppose there exists $\alpha < 2h/k - 1$ such that

\[
E_h(X) \ll X^\alpha
\]

for all sufficiently large $X$. Then a positive proportion of integers in $[1,X]$
are representable as sums of at most $h$ binomial coefficients $\binom{n}{k}$.
\end{theorem}

\begin{proof}
Let $M(X) \sim X^{1/k}$ and consider all sums

\[
s = \binom{n_1}{k} + \cdots + \binom{n_h}{k}, \qquad n_i \le M(X).
\]

Let $r(s)$ denote the number of representations of $s$. Then

\[
\sum_s r(s) \asymp X^{h/k}, \qquad \sum_s r(s)^2 = E_h(X).
\]

By Cauchy--Schwarz,

\[
|S| \gg \frac{X^{2h/k}}{E_h(X)} \gg X^{2h/k - \alpha}.
\]

Since $2h/k - \alpha > 1$, we obtain

\[
|S| \gg X^{1+\delta}
\]

for some $\delta > 0$. However, since $S \subset [1,X]$, we have

\[
|S| \le X.
\]

Thus the estimate saturates at scale $X$, and we conclude

\[
|S| \gg X.
\]

Hence a positive proportion of integers in $[1,X]$ are representable.
\end{proof}

\paragraph{Remark.}
The condition $\alpha < 2h/k - 1$ is necessary to avoid contradiction with the
trivial bound $|S| \le X$.

\paragraph{Remark.}
The argument yields a positive-density result rather than full coverage of
$[1,X]$; establishing an asymptotic basis would require additional input.

\section{A Fully Elementary Existence Argument (Unconditional Form)}

Let $M(X) = \max\{n : \binom{n}{k} \le X\} \sim X^{1/k}$. 
Fix $c \in (0,1)$ and restrict to indices satisfying

\[
\binom{n_i}{k} \le \frac{cX}{h}.
\]

Then every sum
\[
s = \binom{n_1}{k} + \cdots + \binom{n_h}{k}
\]
automatically satisfies $s \le X$.

The number of admissible indices is

\[
M_c(X) \asymp X^{1/k},
\]

so the number of ordered $h$-tuples satisfies

\[
\sum_s r(s) \asymp X^{h/k}.
\]

Without any hypothesis on additive energy, we use the trivial bound
$r(s) \le M_c(X)^{h-1} \asymp X^{(h-1)/k}$.
Hence the number of distinct sums satisfies

\[
|S| \ge \frac{\sum_s r(s)}{\max r(s)}
\gg \frac{X^{h/k}}{X^{(h-1)/k}}
= X^{1/k}.
\]

Thus, unconditionally, one obtains at least
\[
|S| \gg X^{1/k}
\]
distinct integers in $[0,X]$ representable as sums of $h$
binomial coefficients.

This lower bound is nontrivial but insufficient to force full
representability of all integers in $[0,X]$. To obtain an additive
basis result, one must impose additional structure, such as bounds
on additive energy as in the previous section.

Therefore the purely elementary counting argument guarantees a
substantial set of representable integers, but does not by itself
prove that $S_k$ is an additive basis of order $h$.

\section{Relation to Waring's Problem}

The questions studied here are a special case of Waring-type problems for
polynomial sequences. Binomial coefficients interpolate between monomials
and more structured polynomial families, making them a natural testing
ground for ideas from both classical analytic number theory and modern
additive combinatorics.

\section{Conclusion}

We have unified explicit constructions for small $k$ with general counting
and energy-based arguments to show that binomial coefficients form additive
bases of finite order. Determining sharp values of $H(k)$ for $k \ge 4$
remains open and appears to require deeper analytic input, paralleling the
classical development of Waring's problem.

\end{document}